\documentclass[12pt,leqno]{amsart}
\usepackage{amsfonts}
\usepackage{amssymb}
\usepackage{amsthm}
\usepackage{enumitem}
\usepackage{hyperref}
\usepackage[english]{babel}

\usepackage{graphicx}

\usepackage{float}
\graphicspath{ {./} }

\def\dist{{\rm dist}}

\def\ds{\displaystyle}

\newtheorem{thm}{Theorem}
\newtheorem{lem}{Lemma}

\newtheorem{ex}{Example}

\newtheorem{prop}{Proposition}

\newtheorem{defn}{Definition}
\newtheorem{cor}{Corollary}

\begin{document}
	\title{On Necessary and Sufficient Conditions for Fixed Point Convergence: A Contractive Iteration Principle}
	\author{Vasil Zhelinski}
	\email{vasil\_zhelinski@uni-plovdiv.bg, v.zhelinski@gmail.com}
	\maketitle
	
	\section*{Abstract}
		
		While numerous extensions of Banach's fixed point theorem typically offer only sufficient conditions for the existence and uniqueness of a fixed point and the convergence of iterative sequences, this study introduces a generalization grounded in the iterative contraction principle in complete metric spaces. This generalization establishes both the necessary and sufficient conditions for the existence of a unique fixed point to which all iterative sequences converge, along with an accurate error estimate. Furthermore, we present and prove an additional theorem that characterizes the convergence of all iterative sequences to fixed points that may not be unique. Several examples are provided to illustrate the practical application of these results, including a case where the traditional and well-known generalizations of Banach's theorem, such as those by Banach, Kannan, Chatterjea, Hardy-Rogers, Meir-Keeler, and Guseman, are inapplicable.
		
		\noindent Primary 47H10; Secondary 54H25, 37C25, 65J15, 47H09.\\
	\section{Introduction and Preliminaries}
	
		The concept of fixed points is crucial for addressing equations of the form $Tx=x$, where $T:X\to X$ represents a self-map on a set $X$, typically a subset of a metric or normed space. The fixed-point theory is rapidly expanding in the field of mathematics. Numerous researchers are engaged in this area because fixed-point theory finds applications in various disciplines, including mathematics, biology, engineering, economics, computer science, and machine learning. Problems such as variational inclusion, minimization, equilibrium, variational inequality, and optimization can often be transformed into fixed point problems. A notable area in fixed-point theory is the examination of contraction maps.
	
		Instances of the application of results concerning contraction maps can be found in different scientific fields, such as \cite{ikonomika} in economics, \cite{fizika} in physics, and \cite{ai} in machine learning.
	
		The exploration of contraction maps begins with Banach's fixed-point theorem \cite{BANACH}. This theorem asserts that if $(X,\rho)$ is a complete metric space and $T: X\to X$ is a mapping for which there exists a constant $k\in[0;1)$ such that for all $x,y\in X$, the inequality $\rho(Tx,Ty)\leq k \rho(x,y)$ is satisfied, then there is a unique fixed point of $T$ in $X$. Moreover, the sequences generated by iterating $T$ (denoted as $\{T^nx\}_{n=0}^\infty$), starting from an initial point $x\in X$, converge to this fixed point. Numerous extensions of this fundamental theorem have been developed over time. These include the examination of various related spaces, such as Banach spaces and b-metric spaces introduced in \cite{B-MS-INTRODUCED} and briefly analyzed in the context of contraction maps in \cite{B-MS-1,B-MS-2,B-MS-3}, partial metric spaces introduced in \cite{PMS-INTRODUCTED} and briefly analyzed in the context of contraction maps in \cite{PMS-1,PMS-2,PMS-3}, and modular function spaces, initially utilized in \cite{KKR} and subsequently studied in \cite{Ilchev20172873,Khamsi-Kozlowski,Kozlowski,ZLATANOV-m}, among others. Another approach to generalizing Banach's theorem involves altering the contraction condition. Examples of such generalizations include the contraction maps of Kannan \cite{KANAN}, Chatterjea \cite{CHATTERJEA}, Hardy-Rogers \cite{HARDY-ROGERS-0,HARDY-ROGERS}, Meir Keeler \cite{MK}, and others. The original contraction condition of Hardy-Rogers \cite{HARDY-ROGERS-0} necessitates the existence of $0\leq k_i<1$ for $i\in\{1,2,3,4,5\}$, such that $\ds \sum_{i=1}^5 k_i<1$, and requires that every $x$ and $y$ in the considered metric spaces satisfy this condition.
		
		$$\rho(Tx,Ty)\leq k_1\rho(x,y)+k_2\rho(x,Tx)+k_3\rho(y,Ty)+k_4\rho(x,Ty)+k_5\rho(y,Tx),$$
		
		where $T$ is a self-map. Nonetheless, an equivalent contraction condition with fewer parameters was later identified in \cite{HARDY-ROGERS}, as demonstrated in the following theorem.
		\begin{thm}[\cite{HARDY-ROGERS}]\label{HR-thm}
			Consider a complete metric space $(X,\rho)$ and a mapping $T:X\to X$. Suppose there exist constants $0\leq k_i<1$ for $i\in \{1,2,3\}$ such that $k_1+k_2+k_3<1$, and the inequality
            \begin{equation}\label{HR-contraction}
            \rho(Tx,Ty)\leq k_1\rho(x,y)+\frac{k_2}{2}(\rho(x,Tx)+\rho(y,Ty))+\frac{k_3}{2}(\rho(x,Ty)+\rho(y,Tx)),
            \end{equation}
            holds for all $x,y\in X$.
            Then, $T$ possesses a unique fixed point within $X$, and for any $x\in X$, the sequence $\{T^nx\}_{n=1}^\infty$ converges to the fixed point of $T$.
		\end{thm}
		
		It is widely recognized that a map classified as a contraction according to Banach is also considered a contraction in the Hardy–Rogers sense with $k_{2}=k_{3}=0$. In a similar vein, if a map is a contraction as per Kannan's definition, it qualifies as a contraction in the Hardy–Rogers framework with $k_{1}=k_{3}=0$. Likewise, if a map is a contraction in the Chatterjea sense, it is a contraction in the Hardy–Rogers sense with $k_{1}=k_{2}=0$.
        
        The Meir-Keeler fixed point theorem is stated as follows:</div>
		
		\begin{thm}[\cite{MK}]\label{meir-keeler-thm}
			Consider $(X,\rho)$ as a complete metric space and let $T:X\to X$. For each $\varepsilon>0$, there exists a $\delta>0$ such that 
            \begin{equation}\label{meir-keeler-contraction}
                \text{ if }\varepsilon\leq \rho(x,y)<\varepsilon+\delta \text{ then it follows that }\rho(T(x),T(y))<\varepsilon
            \end{equation}
            for all $x,y\in X$. Consequently, $T$ possesses a unique fixed point in $X$, and the sequence $\{T^nx\}_{n=1}^\infty$ converges to this fixed point for all $x\in X$.
		\end{thm}
		
		The concept of a fixed point can be extended in various ways. For instance, coupled fixed points, as described in \cite{COUPLEDFP}, are defined for any set $X$ and a two-valued map $T:X\times X\to X$ as $x,y\in X$ such that $T(x,y)=x$ and $T(y,x)=y$. Similarly, best proximity points, introduced in \cite{BPP}, are defined for a metric space $(X,\rho)$ and a map $T:X\to X$ as $x\in X$ such that $\rho(x,Tx)=\min{\rho(u,Tu):u\in X}$. These extensions of the fixed-point concept lead to analogous extensions of Banach's fixed-point theorem. Some extensions of the results for coupled fixed points can be found in \cite{CFP-1,CFP-2,CFP-3}. The extensions that yield results for the best proximity points are linked to a particular type of mapping known as $p$-cyclic maps, introduced in \cite{PCYCLIC}. These are defined as mappings $T$ with a domain that is a union of a finite sequence of $p$ sets $\{A_n\}_{n=1}^{p}$, where each set is mapped into the next, and the last set is mapped into the first. Formally, this is expressed as $\ds T:\bigcup_{n=1}^{p}A_n\to \bigcup_{n=1}^{p}A_n$ with $T(A_n)\subseteq A_{n+1}$ for $n\in\{1,2,3,\dots, p-1\}$ and $T(A_p)\subseteq A_1$. The results concerning the best proximity points of $p$-cyclic maps are detailed in \cite{P-CYCLIC-1,P-CYCLIC-2,P-CYCLIC-3}.
		
		By integrating multiple approaches to generalize Banach's theorem, one can also achieve valid extensions.
		
		A recent extension of Banach's theorem involves a theorem in which the conditions of a contractive iterate nature are applicable. The exploration of this contractive condition type starts with \cite{CI-Introduction-1}, in which Bryant establishes and demonstrates the conjunction theorem as follows:
		\begin{thm}[\cite{CI-Introduction-1}]
			Consider a complete metric space $(X,\rho)$ and a self-map $T:X\to X$. Suppose there exists a constant $k\in [0;1)$ and an integer $n\in \{1,2,3,\dots\}$ such that for all $u,v\in X$, the following inequality holds:
            \begin{equation}\label{in-e-1}
            \rho(T^nu,T^nv)\leq k \rho(u,v).
            \end{equation}
            Under these conditions, $T$ has a single fixed point in $X$, and the sequences generated by iterating $T$ converge to this fixed point from any starting point in $X$.
		\end{thm}
		
		In \cite{CI-Introduction-2}, Sehgal later derived the same corollaries by ensuring that $n=n(u)$ is unique for each $u$, while also requiring that the constraint $T$ remains continuous. The theorem presented is as follows:
		\begin{thm}[\cite{CI-Introduction-2}]\label{CI-thm}
			Consider a complete metric space $(X,\rho)$ and a continuous self-map $T:X\to X$. Suppose there exists a constant $k\in [0;1)$ such that for any $u,v\in X$, the inequality $\rho(T^{n(u)}u,T^{n(u)}v)\leq k \rho(u,v)$ holds, where $n(u)\in \{1,2,3,\dots\}$ is determined solely by $u$.
            
            In this scenario, $T$ has a single fixed point in $X$, and the sequences generated by iterating $T$ converge to this fixed point, regardless of the starting point in $X$.
		\end{thm}
		
		Shortly after, in \cite{CI-Introduction-3}, Guseman eliminated the need for $T$ to be continuous. Recent research in this area includes \cite{CI-1,CI-2,CI-3,CI-4}.
		
		It is important to note that the generalizations of Banach's theorem mentioned so far offer a condition that is sufficient but not necessary for ensuring a unique fixed point on a map and the convergence of all iterative sequences to it. In this study, we extend Banach's theorem by utilizing the iterative contraction principle, which establishes a necessary and sufficient condition for the presence of a unique fixed point to which all iterative sequences converge, and we also determine the error margin. Additionally, we develop and demonstrate a theorem that provides a necessary and sufficient condition for the convergence of all iterative sequences to fixed points, which may not be unique. Finally, we present our findings through four examples: one in which the fixed point is not immediately apparent and we calculate the error, another in which the most common generalizations of Banach's fixed point theorem are inapplicable, a third with multiple fixed points, and a fourth in which we demonstrate how our results can be used to rule out convergence to a fixed point.
		
		In this study, we adopt the following notation:  
		$\mathbb{Z}$ denotes the set of integers,  
		$\mathbb{N}$ stands for the set of positive natural numbers, and $\mathbb{N}_0$ refers to $\mathbb{N}$ together with zero.  
		The set of real numbers is denoted as $\mathbb{R}$.  
		
		For $x \in \mathbb{R}$, the symbol $\lfloor x \rfloor$ represents the greatest integer not exceeding $x$ (the floor of $x$), while $\lceil x \rceil$ denotes the least integer not smaller than $x$ (the ceiling of $x$).  
		
		A pair $(X,\rho)$ is understood as a metric space with metric $\rho$, and $(X,\|\cdot\|)$ indicates a normed space with norm $\|\cdot\|$.  
		
		Throughout the paper, we shall also use the notation $\ell_{1}$ for the classical Banach space of absolutely summable sequences.  
		More precisely, $\ell_{1}$ has the canonical basis $\{e_k\}_{k=1}^\infty$, where  
		\[
		e_k = (\underbrace{0,0,\dots,0,1}_{k},0,\dots).
		\]
		The space $\ell_{1}$ can be described as
		\[
		\ell_{1} = \left\{ \sum_{k=1}^\infty x_k e_k : \{x_{k}\}_{k=1}^\infty \subset \mathbb{R},\; \sum_{k=1}^{\infty} |x_{k}| < \infty \right\},
		\]
		equipped with the norm
		\[
		\left\|\sum_{k=1}^\infty x_k e_k\right\|_{1} = \sum_{k=1}^{\infty} |x_{k}|.
		\]
		In this form, $\ell_{1}$ is a normed linear space that is complete with respect to $\|\cdot\|_{1}$.
		
	\section{Main Result}	
	
	\begin{defn}
		Consider a metric space $(X,\rho)$ and a mapping $T:X\to X$. Let $k$ be a value in the interval $[0,1)$. The following conditions are satisfied:
        
        \begin{itemize}
        \item For any distinct points $x,y\in X$, there exists an integer $N_y^x\in \mathbb{N}$ such that 
        \begin{equation}\label{contraction1}
        \rho(T^{N_y^x}x,T^{N_y^x}y)\leq k\rho(x,y).
        \end{equation}
        \item For each $x\in X$, there is an integer $N_x^x\in \mathbb{N}$ such that \begin{equation}\label{contraction2} 
        \rho(T^{N_x^x}x,T^{N_x^x+n}x)\leq k\rho(x,T^nx),
        \end{equation}
        for every $n\in \mathbb{N}$.
        \end{itemize}
        In this case, $T$ is a universal iterative contraction map.
	\end{defn}
	
	\begin{thm}\label{t1}
		Consider a complete metric space $(X,\rho)$, a constant $k$ within the interval $(0;1)$, and a mapping $T:X\to X$. The following two statements are equivalent,
		\begin{enumerate}[label=(stat-\arabic*)]
			\item\label{st-1} $T$ is a universal iterative contraction map, with $k$.
			\item\label{st-2} There exists an $\alpha\in X$, which serves as the sole fixed point of $T$ within $X$, and for every $x\in X$, the limit $\ds\lim_{n\to\infty}T^nx=\alpha$ holds.
		\end{enumerate}
	\end{thm}
	
	\begin{thm}\label{t3}
		Consider a metric space $(X,\rho)$ and a mapping $T:X\to X$. The following two statements are equivalent,
		\begin{enumerate}[label=(stat-\arabic*)]
			\item\label{t3st-1} There exists a closed cover $\mathcal{C}$ of $X$ such that for every $S\in \mathcal{C}$, the condition $T(S)\subseteq S$ is satisfied. Additionally, $(S,\rho)$ forms a complete subspace of $X$, and within $S$, the map $T$ acts as a universal iterative contraction map for a given $k\in (0,1)$.
			\item\label{t3st-2} The sequence $\{T^nx\}_{n=0}^\infty$ converges to a fixed point of $T$ in $X$ for each $x\in X$.
		\end{enumerate}
	\end{thm}
	
	\begin{thm}\label{t2}
		Consider a complete metric space $(X,\rho)$ and a universal iterative contraction map $T$. Let $x$ be an element of $X$, and define $\ds\alpha=\lim_{n\to \infty}T^nx$. The sequence $\{p_n\}_{n=0}^\infty\in \mathbb{N}_0$ is specified by $p_0=0$ and $p_{n+1}=p_n+N_{T^{p_n}x}^{T^{p_n}x}$. For each $n\in \mathbb{N}_0$.
        
        Then, the following inequalities are satisfied: $$\rho(\alpha,T^{p_n}x)\leq k^n M_x\ {\rm and}\ \sup_{i\geq p_n}\rho(\alpha,T^ix)\leq 2k^n M_x,$$ where
		$$M_x=\max\left(\{\rho(x,T^{m}x):m\in\{1,2,\dots N_x^x-1\}\}\cup\left\{\frac{1}{1-k}\rho(T^{N_x^x}x,x)\right\}\right).$$
	\end{thm}
	
	\section{Auxiliary Results}
	
	\begin{lem}\label{l1}
		Consider a metric space $(X,\rho)$, a universal iterative contraction map $T$, and an element $x$ belonging to $X$.
		
		Then, the sequence $\{T^nx\}_{n=0}^\infty$ is a Cauchy sequence. Specifically, for each $n\in \mathbb{N}_0$, the inequality \\ $\ds\sup_{i\geq p_n}\rho(T^{p_n}x,T^ix)\leq k^nM_x$ is satisfied, where the sequence $\{p_n\}_{n=0}^\infty\subseteq \mathbb{N}_0$ is defined by $p_0=0$, $p_{n+1}=p_n+N_{T^{p_n}x}^{T^{p_n}x}$, and $$M_x=\max\left(\{\rho(x,T^{m}x):m\in\{1,2,\dots N_x^x-1\}\}\cup\left\{\frac{1}{1-k}\rho(T^{N_x^x}x,x)\right\}\right).$$
	\end{lem}
	\begin{proof}
		Consider $x\in X$. Utilizing (\ref{contraction2}) alongside the triangle inequality, we deduce that for any $n\in \mathbb{N}$, the following holds:
        $$\rho(x,T^{N_x^x+n}x)\leq\rho(T^{N_x^x}x,T^{N_x^x+n}x)+\rho(T^{N_x^x}x,x)\leq k\rho(x,T^nx)+\rho(T^{N_x^x}x,x).$$
        By repeatedly applying this inequality, we derive the sequence of inequalities 
        $$ 
        \begin{array}{lll}
        \rho(x,T^{nN_x^x+m}x)&\leq& k\rho(x,T^{(n-1)N_x^x+m}x)+\rho(T^{N_x^x}x,x)\\
        &\leq&k^2\rho(x,T^{(n-2)N_x^x+m}x)+(k+1)\rho(T^{N_x^x}x,x)\\ &&\dots\\
        &\leq&\ds k^{n-1}\rho(x,T^{N_x^x+m}x)+\frac{1-k^{n-1}}{1-k}\rho(T^{N_x^x}x,x)\\
        &\leq&\ds k^{n}\rho(x,T^{m}x)+\frac{1-k^{n}}{1-k}\rho(T^{N_x^x}x,x)\\
        &\leq&\ds \max\left\{\rho(x,T^{m}x),\frac{1}{1-k}\rho(T^{N_x^x}x,x)\right\},\\
        \end{array}
        $$
        is valid for any $n,m\in \mathbb{N}_0$. Consequently, for each $m\in \mathbb{N}_0$, we can express 
        $$\ds\sup_{n\in \mathbb{N}_0}\rho(x,T^{nN_x^x+m}x)\leq\max\left\{\rho(x,T^{m}x),\frac{1}{1-k}\rho(T^{N_x^x}x,x)\right\}.$$
        Then, considering $\ds \{T^{n}x\}_{n=0}^\infty=\bigcup_{m=0}^{N_x^x-1}\{T^{nN_x^x+m}x\}_{n=0}^\infty$, we conclude that there exists $\ds \sup_{n\in \mathbb{N}_0}\rho(x,T^nx)\leq M_x<\infty$.
		
		From (\ref{contraction2}), it is evident that for every $n,i\in \mathbb{N}_0$, the inequality $$\rho(T^{p_{n+1}}x,T^{p_{n+1}+i}x)=\rho(T^{N_{T^{p_{n}}x}^{T^{p_{n}}x}}T^{p_{n}}x,T^{N_{T^{p_{n}}x}^{T^{p_{n}}x}+i}T^{p_n}x)\leq k\rho(T^{p_{n}}x,T^{p_{n}+i}x)$$
        holds true. By applying this $n$ times in succession, we find that for each $n\in \mathbb{N}_0$, the following is valid: 
        $$ \begin{array}{lll} \ds\sup_{i\geq p_n}\rho(T^{p_n}x,T^ix)&=&\ds\sup_{i\in \mathbb{N}_0}\rho(T^{p_n}x,T^{p_n+i}x)\\ &\leq&\ds k\sup_{i\in \mathbb{N}_0}\rho(T^{p_{n-1}}x,T^{p_{n-1}+i}x)\\ &\leq&\ds k^2\sup_{i\in \mathbb{N}_0}\rho(T^{p_{n-2}}x,T^{p_{n-2}+i}x)\\ &&\dots\\ &\leq&\ds k^n\sup_{i\in \mathbb{N}_0}\rho(x,T^ix)\\ &\leq&\ds k^nM_x.\\ \end{array} $$
        Given that $k\in [0,1)$, it follows that for any $\varepsilon>0$, there exists an $n\in \mathbb{N}_0$ such that $\ds\sup_{i\geq n}\rho(T^{n}x,T^ix)\leq \varepsilon$. Consequently, the sequence $\{T^nx\}_{n=0}^\infty$ is a Cauchy sequence.
	\end{proof}
	\begin{lem}\label{l2}
		Consider $(X,\rho)$ as a complete metric space, with $T$ being a universal iterative contraction mapping.
		
		Then, there exists an $\alpha \in X$ such that for every $x\in X$, the limit $\ds\lim_{n\to\infty}T^nx$ equals $\alpha$.
	\end{lem}
	\begin{proof}
		To demonstrate the assertion, it suffices to show that for any $x,y\in X$, the equation $\ds\lim_{n\to\infty}T^nx=\lim_{n\to\infty}T^ny=u\in X$ is valid.
        
        Consider $x,y\in X$ and define the sequence $\{p_n\}_{n=0}^\infty$ by setting $p_0=0$ and $p_{n+1}=p_n+N_{T^{p_n}y}^{T^{p_n}x}$. From (\ref{contraction1}), it is evident that for each $n\in \mathbb{N}_0$, the following holds: $$\rho(T^{p_{n+1}}x,T^{p_{n+1}}y)=\rho(T^{N_{T^{p_{n}}y}^{T^{p_{n}}x}}T^{p_{n}}x,T^{N_{T^{p_{n}}y}^{T^{p_{n}}x}}T^{p_n}y)\leq k\rho(T^{p_{n}}x,T^{p_{n}}y).$$ By applying this inequality consecutively $n$ times, we derive $\rho(T^{p_{n}}x,T^{p_{n}}y)\leq k^n\rho(x,y)$ for every $n\in \mathbb{N}_0$. Consequently, $\ds\lim_{n\to\infty}\rho(T^{p_n}x,T^{p_n}y)=0$. Utilizing lemma \ref{l1} and the completeness of $(X,\rho)$, we find $\ds\lim_{n\to\infty}T^nx=\lim_{n\to\infty}T^{p_n}x=x'\in X$ and $\ds\lim_{n\to\infty}T^ny=\lim_{n\to\infty}T^{p_n}y=y'\in X$. Given $\ds\lim_{n\to\infty}\rho(T^{p_n}x,T^{p_n}y)=0$ and the continuity of the metric, it follows that $x'=y'$.
	\end{proof}
	\begin{lem}\label{l3}
		Consider a complete metric space $(X,\rho)$ and a universal iterative contraction map $T$. If $x$ is an element of $X$ and $\ds\lim_{n\to\infty}T^nx=\alpha$, then $T\alpha=\alpha$.
	\end{lem}
	\begin{proof}
		Assuming $x\in X$ and $\ds\lim_{n\to\infty}T^nx=\alpha$, lemma \ref{l2} indicates that $\alpha \in X$ and $\ds\lim_{n\to\infty}T^n\alpha=\alpha$. According to (\ref{contraction2}), for every $n\in \mathbb{N}_0$, the inequality $0\leq\rho(T^{N_\alpha^\alpha}\alpha,T^{N_\alpha^\alpha+n}\alpha)\leq k\rho(\alpha,T^n\alpha)$ holds. Given $\ds\lim_{n\to\infty}T^n\alpha=\alpha$, it follows that $\ds\lim_{n\to\infty}\rho(T^{N_\alpha^\alpha}\alpha,T^{N_\alpha^\alpha+n}\alpha)=0$, since the metric function is continuous. Consequently, $\ds\rho(T^{N_\alpha^\alpha}\alpha,\alpha)=\lim_{n\to\infty}\rho(T^{N_\alpha^\alpha}\alpha,T^{N_\alpha^\alpha+n}\alpha)=0$, implying $T^{N_\alpha^\alpha}\alpha=\alpha$. Suppose $T\alpha=\beta\neq \alpha$. Then, using $T^{N_\alpha^\alpha}\alpha=\alpha$, it follows that $T^{nN_\alpha^\alpha+1}\alpha =\beta$ for any $n\in\mathbb{N}_0$. However, since $\beta\neq \alpha$, it becomes evident that $\ds\lim_{n\to\infty}T^{n}\alpha= \alpha$ is false, contradicting $\ds\lim_{n\to\infty}T^{n}\alpha= \alpha$. Therefore, $T\alpha=\alpha$.
	\end{proof}
	\begin{cor}\label{c1}
		Consider a metric space $(X,\rho)$ and a sequence $\{x_n\}_{n=0}^\infty$ within $X$ such that $\ds\lim_{n\to\infty}x_n=x'$. For any point $y$ in $X$ where $y$ is not equal to $x'$, it follows that $\inf\{\rho(y,x_n):n\in \mathbb{N}_0,x_n\neq y\}$ is greater than zero.
	\end{cor}
	\begin{proof}
		Consider $y\in X$ with $y\neq x'$. This implies $\rho(y,x')=2\delta>0$. Given that $\ds\lim_{n\to\infty}x_n=x'$, there must be some $N\in \mathbb{N}$ such that for all $n\geq N$, the inequality $\rho(x_n,x')<\delta$ holds. Applying the triangle inequality and knowing $\rho(y,x')=2\delta$, we deduce that
        \begin{equation}\label{c1e1}
        \inf_{n\geq N}\rho(y,x_n)\geq \delta.
        \end{equation}
        There is a minimum value $\min\{\rho(y,x_n):n\in \{1,2,\dots N-1\},x_n\neq y\}=M>0$. Consequently, using (\ref{c1e1}), it follows that $\ds \inf\{\rho(y,x_n):n\in \mathbb{N}_0,x_n\neq y\}\geq \min\{\delta,M\}$. Since both $\delta>0$ and $M>0$, we conclude \\
        $\ds \inf\{\rho(y,x_n):n\in \mathbb{N}_0,x_n\neq y\}>0$.
	\end{proof}
	\begin{lem}\label{l4}
		Consider a metric space $(X,\rho)$, where $T:X\to X$ and $\alpha\in X$ is the unique fixed point of $T$ in $X$. Given $k\in(0;1)$ and the condition $\ds\lim_{n\to\infty}T^nx=\alpha$ for all $x\in X$, it follows that $T$ is a universal iterative contraction map with constant $k$.
	\end{lem}
	\begin{proof}
		Consider $x,y\in X$. There are three cases:
		\begin{enumerate}[label=case-\arabic*.]
			\item $y\neq \alpha$. Given the assumption $\ds\lim_{n\to\infty}T^nx=\alpha$ and corollary \ref{c1}, it can be deduced that
            \begin{equation}\label{t1e1}
            \inf\{\rho(y,T^nx):n\in \mathbb{N}_0,T^nx\neq y\}=M>0.
            \end{equation}
            According to the assumption $\ds\lim_{n\to\infty}T^ny=\lim_{n\to\infty}T^nx=\alpha$, there exists an $N_y^x\in \mathbb{N}$ such that for any $n\in \mathbb{N}_0$, the inequality $\rho(T^{N_y^x}y,T^{N_y^x+n}x)\leq \rho(T^{N_y^x}y,\alpha)+\rho(\alpha,T^{N_y^x+n}x)\leq kM$ holds. By utilizing this and (\ref{t1e1}), it becomes evident that the inequality
            \begin{equation}\label{contraction}
            \rho(T^{N_y^x}y,T^{N_y^x+n}x)\leq k\rho(y,T^nx),
            \end{equation}
            is valid for every $n\in \mathbb{N}_0$ where $y\neq T^nx$.
            
            In the case where $y= T^nx$, it is straightforward to verify that (\ref{contraction}) also applies for $N_y^x$.						Consequently, (\ref{contraction1}) and (\ref{contraction2}) are valid for $N_y^x$ and all $n\in \mathbb{N}_0$.
			
			\item $y= \alpha$ and $x\neq \alpha$. According to the assumption, $\ds\lim_{n\to\infty}T^nx=\alpha$. Therefore, there exists an $N_y^x\in \mathbb{N}$ such that the inequality $\rho(y,T^{N_y^x}x)\leq k\rho(y,x)$ holds true. Consequently, by assuming that $\alpha$ is a fixed point of $T$, we can conclude that inequality (\ref{contraction1}) holds.
			
			\item $x=y=\alpha$. Here, given that $\alpha$ is a fixed point of $T$, it is straightforward to verify that (\ref{contraction2}) is satisfied for $N_y^x=1$ and for every $n\in \mathbb{N}_0$.
		\end{enumerate}
		In all three scenarios, we find that there is an $N_y^x\in \mathbb{N}$ such that:
        \begin{itemize}
        \item (\ref{contraction1}) is satisfied if $x\neq y$.
        \item (\ref{contraction2}) holds true for every $n\in \mathbb{N}_0$ when $x=y$. \end{itemize}
        Consequently, $T$ serves as a universal iterative contraction map with a constant $k$.
	\end{proof}
	
	\section{proof of main result}
	
	\subsection{proof of Theorem \ref{t1}}
	
	According to Lemma \ref{l4}, we deduce that \ref{st-1} is a consequence of \ref{st-2}.
    
    Now, suppose that \ref{st-1} is true. From Lemma \ref{l2}, we find that there exists an $\alpha \in X$ such that $\ds\lim_{n\to\infty}T^nx=\alpha$ for every $x\in X$. By applying Lemma \ref{l3}, we determine that $\alpha$ is the fixed point of $T$. The only remaining task is to demonstrate that $\alpha$ is the unique fixed point of $T$. Assume that there exists a $\beta\in X$ such that $T\beta=\beta$ and $\beta\neq \alpha$. In this case, we observe that $\ds\lim_{n\to\infty}T^n\beta=\beta$. This limit, along with the conditions that $\beta\in X$ and $\beta\neq \alpha$, contradicts Lemma \ref{l2}. Consequently, \ref{st-2} is derived from \ref{st-1}.
	
	\subsection{proof of Theorem \ref{t3}}
	By applying Theorem \ref{t1}, it becomes apparent that \ref{t3st-2} is a consequence of \ref{t3st-1}.
	
	Assuming that \ref{st-2} is valid. According to this assumption, for each $x\in X$, the limit $\ds\lim_{n\to \infty}T^nx=l(x)$ exists and is an element of $X$, with the property that $Tl(x)=l(x)$. Define $S(x)=\{T^nx\}_{n=0}^\infty\cup \{l(x)\}$ for every $x\in X$. Consider an arbitrary $x \in X$. Given that $T l(x) = l(x)$, we have
    \begin{equation}\label{t3e1}
    T(S(x))=T(\{T^nx\}_{n=0}^\infty)\cup T(\{l(x)\})=\{T^nx\}_{n=1}^\infty\cup\{l(x)\}\subseteq S(x).
    \end{equation}
    
    Next, we demonstrate that $S(x)$ is compact. Let $\mathcal{P}$ be any open cover of $S(x)$. Choose $P'\in \mathcal{P}$ such that $l(x)\in P'$. Since $\ds\lim_{n\to \infty}T^nx=l(x)$, there exists an $N\in \mathbb{N}$ such that $\{T^nx\}_{n=N}^\infty\subseteq P'$. For each $0 \leq n < N$, select $P_n \in \mathcal{P}$ such that $T^nx\in P_n$. Thus, $\mathcal{P}'=\{P_n\}_{n=0}^{n=N-1}\cup \{P'\}$ forms a finite subcover of $\mathcal{P}$. Consequently, $S(x)$ is compact, making it a closed subset of $(X,\rho)$, and $(S(x),\rho)$ is a complete metric space.
    
    Consider an arbitrary element $u$ of $S(x)$. There are two possibilities: $u\in \{T^nx\}_{n=0}^\infty$ or $u=l(x)$. In both scenarios, it is evident that $\ds\lim_{n\to \infty}T^nu=l(x)$. Therefore, the sequence $\{T^nu\}_{n=0}^\infty$ converges to $l(x)$ for every $u\in S(x)$, and $l(x)$ is a unique fixed point of $T$ in $S(x)$. Indeed, if there exists an $\alpha\in S(x)$ such that $\alpha \neq l(x)$ and $T\alpha=\alpha$, it contradicts the fact that $\ds\lim_{n\to \infty}T^nu=l(x)$ for each $u\in S(x)$. From (\ref{t3e1}) and Lemma \ref{l4}, it follows that $T$ is a universal iterative contraction map on $S(x)$ for any $k\in(0,1)$. Using (\ref{t3e1}), the fact that $x\in S(x)$, and the properties that $S(x)$ is closed and $(S(x),\rho)$ is a complete metric space, we conclude that $\mathcal{C}=\{S(x):x\in X\}$ is a closed cover of $X$ that meets the conditions in \ref{t3st-1}.
	
	\subsection{proof of Theorem \ref{t2}}
	Let $n\in \mathbb{N}_0$. According to Lemma \ref{l1}, the inequality $\ds\sup_{i\geq p_n}\rho(T^{p_n}x,T^ix)\leq k^nM_x$ holds true. This implies that the sequence $\{T^ix\}_{i=p_n}^\infty$ is contained within the ball $B[T^{p_n}x,k^nM_x]$. Given the assumption that $\ds\alpha=\lim_{i\to \infty}T^ix$, and considering that $B[\cdot,\cdot]$ represents a closed ball, it follows that $\alpha$ is an element of $B[T^{p_n}x,k^nM_x]$, meaning $\rho(\alpha,T^{p_n}x)\leq k^n M_x$. Utilizing the inequality $\ds\sup_{i\geq p_n}\rho(T^{p_n}x,T^ix)\leq k^nM_x$, we can deduce that $$\sup_{i\geq p_n}\rho(\alpha,T^ix)\leq \sup_{i\geq p_n}\rho(T^{p_n}x,T^ix)+\rho(\alpha,T^{p_n}x)\leq 2k^nM_x.$$
	
	\section{Examples}
	\begin{prop}\label{prop1}
		Consider $A$ as a real interval and $f:A\to A$ as a function with a continuous derivative on $A$. Let $n\in \mathbb{N}$ and $k\in[0,1)$. It is given that $\ds \left|\prod_{i=0}^{n-1}f'(f^i(x))\right|\leq k$ for all $x\in A$. Consequently, for any $u,v\in A$, the condition in (\ref{in-e-1}) is satisfied.
	\end{prop}
	\begin{proof}
		Suppose there exist $u,v\in A$ such that $|f^n(v)-f^n(u)|> k |v-u|$. We can assume, without loss of generality, that $u<v$. This implies $\ds \left|\frac{f^n(v)-f^n(u)}{v-u}\right|>k$. Consequently, there is a point $\xi\in A$ with $u\leq \xi\leq v$ where $|(f^n(\xi))'|>k$. By applying the chain rule, we have $(f^n(\xi))'=f'(f^{n-1}(\xi))(f^{n-1}(\xi))'$. Therefore, it follows that $\ds \left|\prod_{i=0}^{n-1}f'(f^i(\xi))\right|> k$. This conclusion contradicts our initial assumptions. Hence, for all $u,v\in A$, the inequality $|f^n(v)-f^n(u)|\leq k |v-u|$ must hold.
	\end{proof}
	
	This instance does not have an evident fixed point and demonstrates the error estimate.
	
	\begin{ex}
		We need to determine the fixed points of the mapping $T:\mathbb{R}\to \mathbb{R}$, which is given by $T(x)=\ds 1.5+x -\frac{3}{1+\exp(-x)}+0.5 \cos\left(\frac{x}{2}\right)$.
	\end{ex}
	\begin{proof}
	We can calculate that 
	\begin{itemize}
		\item For all $a$ values less than or equal to -1 and $b$ values greater than or equal to 3, it is possible to express
		\begin{equation}\label{e1e1}
			T[a,b]\subseteq [a,b].
		\end{equation}
		\item There holds 
		\begin{equation}\label{e1e2}
			\sup_{x\in [-1,3]}|T'(x)|\leq 0.62\ and\ \sup_{x\in \mathbb{R}}|T'(x)|\leq 1.25.
		\end{equation}
		\item For all $x$ values greater than 3, the inequality chain
        \begin{equation}\label{e1e3}
        -2.0\leq Tx-x\leq -0.99,
        \end{equation}
        holds true.
		\item For every $x<-1$ there holds
		\begin{equation}\label{e1e4}
			0.99\leq Tx-x\leq2.
		\end{equation}
	\end{itemize}
	Based on (\ref{e1e1}), (\ref{e1e3}), and (\ref{e1e4}), it can be deduced that for any $x$ in the intervals $(-\infty,-1)\cup (3,\infty)$ and for $n$ greater than or equal to $N_x=\inf\{i\in \mathbb{N}:0.99 i\geq \dist(x,[-1,3])\}$, the condition
    \begin{equation}\label{e1e5}
    T^nx\in [-1,3].
    \end{equation}
    holds true.
    
    When $x$ is within the range $[-1,3]$, we set $N_x=0$.
	
	Based on the definition of $N_x$, it can be deduced that if $a\leq -1$, $b\geq3$, and $x$ lies within the interval $[a,b]$, then $N_x$ is at most $\max\{N_a,N_b\}$. Subsequently, by applying (\ref{e1e1}), (\ref{e1e2}), and (\ref{e1e5}), it is evident that for any $a\leq -1$, $b\geq3$, and $x$ in the range $[a,b]$, the inequality $\ds \left|\prod_{i=0}^{P(a,b)-1}T'(T^i(x))\right|\leq 0.62$ holds true, where $P(a,b)$ is defined as $\max\{N_a+\left\lceil\log_{0.62}\left(\frac{0.62}{1.25^{N_a}}\right)\right\rceil,N_b+\left\lceil\log_{0.62}\left(\frac{0.62}{1.25^{N_b}}\right)\right\rceil\}$. It is straightforward to confirm that $T'$ is continuous on $\mathbb{R}$. Utilizing (\ref{e1e3}), (\ref{e1e4}), and proposition \ref{prop1}, we derive that:
    \begin{itemize} 
        \item for any $x,y\in \mathbb{R}$, the condition (\ref{contraction1}) is satisfied with $k=0.62$ and $N_x^y=P(\min\{x,y,-1\},\max\{x,y,3\})$. 
        \item for any $x\in \mathbb{R}$, (\ref{contraction2}) holds true, with $k=0.62$ and $N_x^x=P(\min\{x,-1\},\max\{x,3\})$. 
    \end{itemize} 
    This implies that $T$ functions as an universal iterative contraction map. Applying Theorem \ref{t1}, we conclude that $T$ possesses a unique fixed point $\alpha\in \mathbb{R}$ and $\ds \lim_{n\to \infty}T^nx=\alpha$ for any $x\in \mathbb{R}$. We initiated an approximation of $\alpha$ starting with $u=8.5$. Define the sequence $\{p_n\}_{n=0}^\infty\subset \mathbb{N}_0$ as $p_0=0$, $p_{n+1}=p_n+N_{T^{p_n}u}^{T^{p_n}u}$. Calculations show that $p_1=10, p_2=11, p_3=12 \dots$, and $M_u=20.6474$. Then, by employing Theorem \ref{t2}, we derive
	$$
	\begin{array}{|l|l|l|l|}
		\hline
		n&p_n&T^{p_n}u&|T^{p_n}u-\alpha|\leq\cdot\\\hline
		0&0&8.5&20.6474\\\hline
		1&10&0.654004&12.8014\\\hline
		2&11&0.653772&7.93684\\\hline
		5&14&0.653698&1.89157\\\hline
		10&19&0.653697&0.173293\\\hline
		20&29&0.653697&0.00145445\\\hline
	\end{array}
	$$
\end{proof}
	
	In the following example, we demonstrate that our findings are applicable, whereas the most widely recognized extensions of the Banach's fixed-point theorem are not. We will illustrate that Theorems \ref{HR-thm}, \ref{meir-keeler-thm}, and \ref{CI-thm} (excluding the necessity for $T$ to be continuous \cite{CI-Introduction-3}) are inapplicable because the contraction principles of Banach, Kannan, and Chatterjea are specific instances of the Hardy-Rogers contraction principle.
	
\begin{ex}
	Consider the Banach space $\ell_1$, and define the map $T:\ell_1\to \ell_1$ by the expression $\ds T\left( \sum_{i=1}^\infty x_ie_i\right)= \sum_{i=1}^\infty\left(\frac{1}{\sqrt[i]{2}}x_i+\left(1-\frac{1}{\sqrt[i]{2}}\right)\frac{1}{i^2}\right)e_i$. This map $T$ serves as a universal iterative contraction with $\ds k=\frac{1}{2}$. Furthermore, Theorems \ref{meir-keeler-thm}, \ref{HR-thm}, and \ref{CI-thm} (without the necessity for $T$ to be continuous \cite{CI-Introduction-3}) are not applicable.
\end{ex}
\begin{proof}
	Initially, we demonstrate that $T$ serves as a universal iterative contraction map, with $\ds k=\frac{1}{2}$.
	
	Consider that for any $x,y\in \ell_1$ and $p\in \mathbb{N}$, the following holds: $\ds \|T^px-T^py\|_1=\sum_{i=1}^\infty\frac{|x_i-y_i|}{(\sqrt[i]{2})^p}$. Define $J^y_x\in \mathbb{N}$ such that $\ds \sum_{i=J^y_x+1}^{\infty}|x_i-y_i|\leq\frac{1}{4}\sum_{i=1}^{J^y_x}|x_i-y_i|$. The sequence $\ds \left\{\frac{1}{(\sqrt[i]{2})^p}\right\}_{i=1}^\infty$ is positive, increasing, and has a supremum of 1 for any $p\in \mathbb{N}$. Consequently, we have:
    $$ 
    \begin{array}{lll}
    \|T^px-T^py\|_1&=&\ds\sum_{i=1}^{J^y_x}\frac{|x_i-y_i|}{(\sqrt[i]{2})^p}+\sum_{i=J^y_x+1}^{\infty}\frac{|x_i-y_i|}{(\sqrt[i]{2})^p}\\
    &\leq&\ds\frac{1}{(\sqrt[{J^y_x}]{2})^p}\sum_{i=1}^{J^y_x}|x_i-y_i|+\sum_{i=J^y_x+1}^{\infty}|x_i-y_i|\\
    &\leq&\ds\left(\frac{1}{(\sqrt[{J^y_x}]{2})^p}+\frac{1}{4}\right)\sum_{i=1}^{J^y_x}|x_i-y_i|\\
    &\leq &\ds\left(\frac{1}{(\sqrt[{J^y_x}]{2})^p}+\frac{1}{4}\right)\|x-y\|_1. \end{array}
    $$
    Therefore, by setting $N^y_x=2J^y_x$, we confirm that (\ref{contraction1}) is satisfied.
	
	Consider $x\in \ell_1$. For any $p\in \mathbb{N}_0$ and $n\in \mathbb{N}$, the expression can be represented as
    \begin{equation}\label{e2e5}
    \|T^px-T^{p+n}x\|_1=\sum_{i=1}^{\infty} \frac{1}{(\sqrt[i]{2})^p}\left(1-\frac{1}{(\sqrt[i]{2})^n}\right)\left|x_i-\frac{1}{i^2}\right|.
    \end{equation}
    Define $J_x\in \mathbb{N}$ such that $\ds \sum_{i=J_x}^{\infty}\left|x_i-\frac{1}{i^2}\right|\leq \frac{1}{4}\sum_{i=1}^{J_x-1}\left|x_i-\frac{1}{i^2}\right|$. The sequence $\ds \left\{1-\frac{1}{(\sqrt[i]{2})^n}\right\}_{i=1}^\infty$ is positive and decreasing for any $n\in \mathbb{N}$. Thus,
    $$
    \begin{array}{lll}
    \ds\sum_{i=J_x}^{\infty}\left(1-\frac{1}{(\sqrt[i]{2})^n}\right)\left|x_i-\frac{1}{i^2}\right|&\leq&\ds\left(1-\frac{1}{(\sqrt[J_x]{2})^n}\right)\sum_{i=J_x}^{\infty}\left|x_i-\frac{1}{i^2}\right|\\
    &\leq&\ds\left(1-\frac{1}{(\sqrt[J_x-1]{2})^n}\right)\frac{1}{4}\sum_{i=1}^{J_x-1}\left|x_i-\frac{1}{i^2}\right|\\
    &\leq&\ds\frac{1}{4}\sum_{i=1}^{J_x-1}\left(1-\frac{1}{(\sqrt[i]{2})^n}\right)\left|x_i-\frac{1}{i^2}\right|.
    \end{array}	$$
    Then, considering that the sequence $\ds \left\{\frac{1}{(\sqrt[i]{2})^p}\right\}_{i=1}^\infty$ is positive, increasing, and has a supremum of $1$ for any $p\in \mathbb{N}$, and using (\ref{e2e5}), we derive that
    $$
    \begin{array}{lll}
    &&\ds \|T^px-T^{p+n}x\|_1\\
    &=&\ds\sum_{i=1}^{J_x-1}\frac{1}{(\sqrt[i]{2})^p}\left(1-\frac{1}{(\sqrt[i]{2})^n}\right)\left|x_i-\frac{1}{i^2}\right|+\sum_{i=J_x}^{\infty}\frac{1}{(\sqrt[i]{2})^p}\left(1-\frac{1}{(\sqrt[i]{2})^n}\right)\left|x_i-\frac{1}{i^2}\right|\\
    &\leq&\ds \frac{1}{(\sqrt[J_x-1]{2})^p}\sum_{i=1}^{J_x-1}\left(1-\frac{1}{(\sqrt[i]{2})^n}\right)\left|x_i-\frac{1}{i^2}\right|+\sum_{i=J_x}^{\infty}\left(1-\frac{1}{(\sqrt[i]{2})^n}\right)\left|x_i-\frac{1}{i^2}\right|\\
    &\leq&\ds \left(\frac{1}{(\sqrt[J_x-1]{2})^p}+\frac{1}{4}\right)\sum_{i=1}^{J_x-1}\left(1-\frac{1}{(\sqrt[i]{2})^n}\right)\left|x_i-\frac{1}{i^2}\right|\\
    &\leq&\ds\left(\frac{1}{(\sqrt[J_x-1]{2})^p}+\frac{1}{4}\right)\|x-T^{n}x\|_1.
    \end{array}
    $$
    Consequently, if we set $\ds N_x^x=2 (J_x-1)$, then (\ref{contraction2}) is satisfied for every $n\in \mathbb{N}$.
	
	Therefore, $T$ serves as a universal iterative contraction map, with $\ds k=\frac{1}{2}$.
	
	We now demonstrate that Theorem \ref{meir-keeler-thm} is inapplicable.
	
	Consider $\varepsilon>0$ and $\delta>0$. Let $x\in \ell_1$, and $y\in \ell_1$ be defined as $\ds y=x+\left(\varepsilon+\frac{\delta}{2}\right)e_N$, where $N\in \mathbb{N}$ is chosen such that $\ds\frac{\varepsilon+\frac{\delta}{2}}{\sqrt[N]{2}}\geq \varepsilon+\frac{\delta}{4}$. The existence of such an $N$ is guaranteed by the limit $\ds \lim_{n\to \infty}\frac{\varepsilon+\frac{\delta}{2}}{\sqrt[n]{2}}=\varepsilon+\frac{\delta}{2}>\varepsilon+\frac{\delta}{4}$. Consequently, we have $\varepsilon<\|x-y\|_1<\delta$, yet\\ $\ds\|Tx-Ty\|_1=\sum_{i=1}^\infty\frac{|x_i-y_i|}{\sqrt[i]{2}}=\frac{\varepsilon+\frac{\delta}{2}}{\sqrt[N]{2}}\geq \varepsilon+\frac{\delta}{4}$. This implies that for any $\varepsilon>0$ and $\delta>0$, we can find $x,y \in \ell_1$ such that (\ref{meir-keeler-contraction}) does not hold. In other words, Theorem \ref{meir-keeler-thm} is not applicable in this context.
	
	We will proceed with demonstrating that Theorem \ref{HR-thm} is inapplicable.
	
	Consider $k_1, k_2, k_3 \in [0,1)$ such that $k_1 + k_2 + k_3 = k \in [0,1)$. Define $\ds x = \sum_{i=1}^\infty \frac{e_i}{i^2} \in \ell_1$ and $y = x + e_N$, where $N \in \mathbb{N}$ is chosen so that $\ds \frac{1}{\sqrt[N]{2}} \geq \frac{3}{4} + \frac{1}{4}k$. The existence of such an $N$ is guaranteed by the limit $\ds \lim_{n \to \infty} \frac{1}{\sqrt[n]{2}} = 1 > \frac{3}{4} + \frac{1}{4}k$. Given $Tx = x$, $\|x-y\|_1 = 1$, and $\ds \|Tx - Ty\|_1 = \ds \sum_{i=1}^\infty \frac{|x_i - y_i|}{\sqrt[i]{2}} = \ds \frac{1}{\sqrt[N]{2}} \geq \ds \frac{3}{4} + \frac{1}{4}k$, along with (\ref{e2e5}), we derive the following inequalities: 

    $$
    \begin{array}{ll}
        & \ds k\|x-y\|_1 = k < \frac{3}{4} + \frac{1}{4}k \leq \|Tx - Ty\|_1 \\[15pt]
        & k\|x-Tx\|_1 = 0 < \|Tx - Ty\|_1 \\[15pt]
        & \ds k\|y-Ty\|_1 < \|y-Ty\|_1 = \sum_{i=1}^\infty \left(1 - \frac{1}{\sqrt[i]{2}}\right)\left|y_i - \frac{1}{i^2}\right| = 1 - \frac{1}{\sqrt[N]{2}} \\
        \leq & \frac{1}{4}(1-k) \leq \ds \frac{1}{4} < \frac{3}{4} + \frac{1}{4}k \leq \|Tx - Ty\|_1 \\[15pt]
        & \ds k\|x-Ty\|_1 = k\|y-Ty-e_N\|_1 \\
        = & k\left\|-e_N + \sum_{i=1}^{\infty} \left(1 - \frac{1}{\sqrt[i]{2}}\right)\left(y_i - \frac{1}{i^2}\right)e_i\right\|_1 \\
        = & \ds k\left\|-e_N + \left(1 - \frac{1}{\sqrt[N]{2}}\right)e_N\right\|_1 \\
        = & k\left\|-\frac{1}{\sqrt[N]{2}}e_N\right\|_1 < k < \frac{3}{4} + \frac{1}{4}k \leq \|Tx - Ty\|_1 \\[15pt]
        & k\|y-Tx\|_1 = k\|y-x\|_1 < \|Tx - Ty\|_1.
    \end{array}
    $$

    By summing these five inequalities, each multiplied by $\ds \frac{k_1}{k}, \frac{k_2}{2k}, \frac{k_2}{2k}, \frac{k_3}{2k}, \frac{k_3}{2k}$ respectively, we obtain:

    $$
    \textstyle \|Tx - Ty\|_1 > k_1\|x-y\|_1 + \frac{k_2}{2}(\|x-Tx\|_1 + \|y-Ty\|_1) + \frac{k_3}{2}(\|x-Ty\|_1 + \|y-Tx\|_1).
    $$

    This implies that for any $k_1, k_2, k_3 \in [0,1)$ such that $k_1 + k_2 + k_3 \in [0,1)$, there exist $x, y \in \ell_1$ for which (\ref{HR-contraction}) does not hold. Therefore, Theorem \ref{HR-thm} is not applicable here.
	
	In conclusion, we demonstrate that Theorem \ref{CI-thm} (excluding the necessity for $T$ to be continuous; see \cite{CI-Introduction-3}) is inapplicable.
	
	Consider $k\in[0,1)$, $x\in \ell_1$, and $n(x)\in \mathbb{N}$. Define $y=x+e_N$, where $N\in \mathbb{N}$ is chosen such that $\ds\frac{1}{(\sqrt[N]{2})^{n(x)}}>k$. The existence of such an $N$ is guaranteed by the limit $\ds \lim_{j\to \infty}\frac{1}{(\sqrt[j]{2})^{n(x)}}=1>k$. Given that $\|x-y\|_1=1$, it follows that $\ds\|T^{n(x)}x-T^{n(x)}y\|_1=\sum_{i=1}^\infty\frac{|x_i-y_i|}{(\sqrt[i]{2})^{n(x)}}=\frac{1}{(\sqrt[N]{2})^{n(x)}}>k=k\|x-y\|_1$. This implies that for any $x\in \ell_1$, $k\in[0,1)$, and $n(x)\in \mathbb{N}$, there exists a $y\in \ell_1$ such that $\|T^{n(x)}x-T^{n(x)}y\|_1>k\|x-y\|_1$. Consequently, Theorem \ref{CI-thm} (without the continuity condition for $T$ \cite{CI-Introduction-3}) is inapplicable.
\end{proof}

The following example has several fixed points.

\begin{ex}
	Consider the Cartesian plane $\mathbb{R}_2^2$. Define the set $U=\{(x,y)\in \mathbb{R}_2^2:0\leq x<1,0<y\leq 1\}$, and let $T:\mathbb{R}_2^2\to \mathbb{R}_2^2$ be a mapping of itself, specified by $\ds T(x,y)=(f_y(x),y)$, where $\ds f_y(x)=\frac{x}{1+5 x y \cos\left(\frac{\pi}{2}x\right)}$. It follows that $T(U)\subseteq U$, and for any $u\in U$, the sequence $\{T^nu\}_{n=0}^\infty$ converges to a fixed point of $T$ in $U$.
\end{ex}
\begin{proof}
	Consider the equation
    \begin{equation}\label{e3e1}
        f_y(x)=\frac{1}{\frac{1}{x}+5 y \cos\left(\frac{\pi}{2}x\right)}.
    \end{equation}
    It can be determined that $f_y(x)$ is an increasing function with respect to $x$ for $(x,y) \in \overline{U}$. Given that $f_y(0)=0$ and $f_y(1)=1$ for every $y\in (0,1]$, it follows that $T(U)\subseteq U$.
	
	Consider $y\in (0,1]$ and $r\in (0,1)$. Then, for any $x\in [0,r]$, the following equality holds:
	\begin{equation}\label{e3e2}
		f_y(x)\leq f_{y,r}(x)=\frac{1}{\frac{1}{x}+5 y \cos\left(\frac{\pi}{2}r\right)}.
	\end{equation}
	
	It is easy to determine that $\ds f^n_{y,r}(x)=\frac{1}{\frac{1}{x}+n 5 y \cos\left(\frac{\pi}{2}r\right)}$, and that $f_{y,r}([0,r])\subseteq [0,r]$. By applying (\ref{e3e2}), the inclusion $f_y([0,1])\subseteq [0,1]$, and the fact that $f_y(x)$ is increasing for $(x,y) \in U$, we deduce that for any $n\in \mathbb{N}$ and $x\in [0,r]$, the following holds
    \begin{equation}\label{e3e3}
        \begin{array}{lll}
            0&\leq& f_y^n(x)\leq f_y^{n-1}(f_{y,r}(x))\leq f_y^{n-2}(f^2_{y,r}(x))\leq \dots \leq f^n_{y,r}(x)\\
            &=&\ds\frac{1}{\frac{1}{x}+n 5 y\cos\left(\frac{\pi}{2}r\right)}\leq \frac{1}{n 5 y\cos\left(\frac{\pi}{2}r\right)}.
        \end{array}
    \end{equation}
	
	The images below offer a visual depiction of how $f_y$ and $f_{y,r}$ behave.
	\begin{figure}[H]
		\centering
		\includegraphics[scale=1.4]{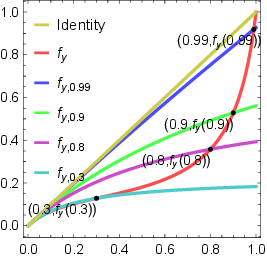}
		\includegraphics[scale=1.4]{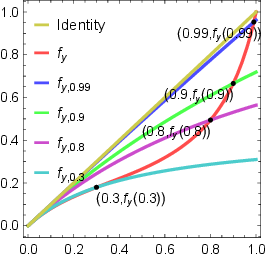}
		\includegraphics[scale=1.4]{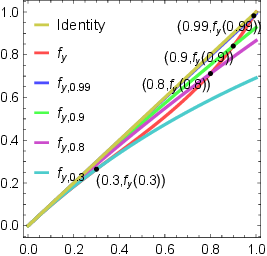}
		\caption{Above in left $y=1$. Above in right $y=0.5$. Below $y=0.1$.}
	\end{figure}
	
	By defining the sets $S^y_r=\{(u,v)\in\mathbb{R}_2^2:0\leq u\leq r,v=y\}$, it becomes evident that $S^y_r$ is a closed subset of $U$ for any $y\in (0,1]$ and $r\in (0,1)$. Consider $y\in (0,1]$ and $r\in (0,1)$. Then, for each $x\in [0,r]$, the following holds
    \begin{equation}\label{e3e3.1}
    0\leq f_y(x)\leq x.
    \end{equation}
    Therefore, $T(S^y_r)\subseteq S^y_r$. Let $x_1,x_2\in [0,r]$, with $x_1\neq x_2$, and define \\ $\ds N(x_1,x_2)=\left\lceil\frac{2}{|x_1-x_2|5 y\cos\left(\frac{\pi}{2}r\right)}\right\rceil$. From (\ref{e3e3}), it follows that \\ $\ds f_y^{N(x_1,x_2)}(x_1),f_y^{N(x_1,x_2)}(x_2)\in \left[0,\frac{|x_1-x_2|}{2}\right]$. Consequently, for any $s_1,s_2\in S^y_r$, the inequality $\ds\|T^{N^{s_1}_{s_2}}s_1-T^{N^{s_1}_{s_2}}s_2\|_2\leq \frac{1}{2}\|s_1-s_2\|_2$ holds true, meaning (\ref{contraction1}) is satisfied, with $\ds k=\frac{1}{2}$ and $N^{s_1}_{s_2}=N(x_1,x_2)$, where $x_1,x_2$ are the first coordinates of $s_1,s_2$, respectively.
	
	Consider $y\in (0,1]$, $r\in(0,1)$, and $x\in [0,r]$. Define\\
    $\ds p(x)=\left\{
    \begin{array}{rr} \ds\left\lceil\frac{2}{|f_y(x)-x|5 y \cos\left(\frac{\pi}{2}r\right)}\right\rceil& :|f_y(x)-x|>0\\
    1& :|f_y(x)-x|=0 \end{array} \right.$
    There are two scenarios.
    \begin{enumerate}[label=case-\arabic*.] 
    
    \item $|f_y(x)-x|>0$, it follows from (\ref{e3e3.1}) that the sequence $\{f_y^n(x)\}_{n=0}^\infty$ is decreasing. Consequently, using (\ref{e3e3}), for every $n\in \mathbb{N}$, the inequality \\
    $\ds 0\leq f_{y}^{p(x)+n}(x)\leq f_{y}^{p(x)}(x)\leq \frac{1}{p(x) 5 y\cos\left(\frac{\pi}{2}r\right)}\leq\frac{1}{2}|f_y(x)-x|$ holds. Additionally, from (\ref{e3e3.1}), we deduce that for every $n\in \mathbb{N}$, $|f_y(x)-x|\leq |f_y^n(x)-x|$. Thus, for every $n\in \mathbb{N}$, it holds that $\ds|f_{y}^{p(x)+n}(x)- f_{y}^{p(x)}(x)|\leq \frac{1}{2}|f_y^n(x)-x|$.
    
    \item $|f_y(x)-x|=0$, it is evident that $f_y^i(x)=x$ for every $i\in \mathbb{N}$. That is, \\ $\ds 0=|f_{y}^{p(x)+n}(x)- f_{y}^{p(x)}(x)|\leq \frac{1}{2}|f_y^n(x)-x|=0$ for every $n\in \mathbb{N}$.
    \end{enumerate}
    In both scenarios, we conclude that $\ds|f_{y}^{p(x)+n}(x)- f_{y}^{p(x)}(x)|\leq \frac{1}{2}|f_y^n(x)-x|$, for every $n\in \mathbb{N}$. Therefore, for any $s\in S^y_r$ and $n\in \mathbb{N}$, the inequality $\ds\|T^{N_s^s}s-T^{N_s^s+n}s\|_2\leq \frac{1}{2}\|s-T^ns\|_2$ holds, which means inequality (\ref{contraction2}) is satisfied, with $k=\frac{1}{2}$ and $N_s^s=p(x)$, where $x$ is the first coordinate of $s$.
    
    Consequently, for each $y\in (0,1]$ and $r\in (0,1)$, $T$ acts as a universal iterative contraction map on $S^y_r$. Therefore, $\mathcal{C}=\{S^y_r:y<0\leq1,0<r<1\}$ forms a closed cover of $U$, meeting the criteria outlined in statement \ref{t3st-1} of Theorem \ref{t3}. Consequently, for any $u\in U$, the sequence $\{T^nu\}_{n=0}^\infty$ converges to a fixed point of $T$ in $U$.
\end{proof}
	
	In the final example, we demonstrate how our findings can be applied to refute the convergence of all iterative sequences towards fixed points.
	
	\begin{ex}
		Consider the real interval $I=\ds\left[\frac{1}{2};1\right)$. Define the function $T:\mathbb{R}\to\mathbb{R}$ by \\
        $\ds Tx=x-2\frac{(x-1)^2}{2x-3}\sin\left(\frac{\pi}{1-x}\right)$. There exists an $\ds x\in I$ such that the sequence $\{T^nx\}_{n=0}^\infty$ does not converge to a fixed point of $T$ within $I$.
	\end{ex}
	\begin{proof}
		We can confirm that $\ds T(I)\subseteq I$. Conversely, the sequence $\{x_n\}_{n=1}^\infty$, defined by $\ds x_n=\frac{4n-1}{4n+1}$, meets the condition $\{x_n\}_{n=1}^\infty\subset I$ and satisfies $Tx_n=x_{n+1}$. Thus, if $U\subset I$, $T(U)=U$, and $x_1\in U$, then $\{x_n\}_{n=1}^\infty\subseteq U$. Given that $1\notin I$ and $1$ is a limit point of $\{x_n\}_{n=1}^\infty$, it is evident that there is no closed subset $S$ of $I$ such that $T(S)=S$ and $x_1\in S$. Consequently, there is no cover of $I$ that fulfills the condition from statement \ref{t3st-1} of Theorem \ref{t3}. Therefore, there exists some $\ds x\in I$ such that the sequence $\{T^nx\}_{n=0}^\infty$ does not converge to a fixed point of $T$ in $I$.
	\end{proof}
	
	\section*{Conclusion}
	
	In this study, we introduce a broadening of Banach's fixed-point theorem, grounded in the iterative contraction principle, which outlines the necessary and sufficient conditions for the presence of a unique fixed point to which all iterative sequences converge, along with an assessment of the related error. Additionally, we developed and demonstrated a theorem that specifies the necessary and sufficient conditions for the convergence of all iterative sequences to fixed points, which may not be unique. Finally, our findings are exemplified through four cases: one in which the fixed point is not immediately apparent and the error is explicitly calculated, another in which the usual generalizations of Banach's theorem do not apply, a third with multiple fixed points, and a fourth showing how our method can be employed to exclude convergence to a fixed point. These results highlight the wider applicability and accuracy of our generalization compared with traditional outcomes.
	
	\section*{Acknowledgments} %
	The author is partially financed by the European Union-NextGenerationEU through the National Recovery and Resilience Plan of the Republic of Bulgaria, project DUECOS BG-RRP-2.004-0001-C01.

	\bibliographystyle{plain}
	\bibliography{bibliography}
\end{document}